\documentclass{tac}

\usepackage{amssymb}
\usepackage[all,cmtip]{xy}
\usepackage{latexsym}
\usepackage{enumitem,verbatim}
\usepackage{hyperref}
\usepackage{stmaryrd}
\usepackage{yhmath}
\usepackage{tikz-cd}

\input diagxy

\author{Hongliang Lai and Dexue Zhang}

\thanks{We would like to thank the support of the National Natural Science Foundation of China (No. 11771310  and No. 11871358).}

\address{School of Mathematics, Sichuan University, Chengdu 610064, China}

\title{Completely distributive enriched categories are not always continuous}

\copyrightyear{2020}



\eaddress{hllai@scu.edu.cn\CR dxzhang@scu.edu.cn}

\newtheorem{thm}{Theorem}
\newtheorem{lem}{Lemma}
\newtheorem{prop}{Proposition}
\newtheorem{cor}{Corollary}

\newtheoremrm{ques}{Question}
\newtheoremrm{defn}{Definition}
\newtheoremrm{exmp}{Example}
\newtheoremrm{rem}{Remark}

\newtheoremrm{note}{Note}
\newtheoremrm{case}{Case}
\newtheoremrm{con}{Convention}

\DeclareMathOperator{\with}{\&}
\DeclareMathOperator{\colim}{colim}

\def\oto{{\bfig\morphism<180,0>[\mkern-4mu`\mkern-4mu;]\place(86,0)[\circ]\efig}}

\newcommand\thda{\mathrel{\rotatebox[origin=c]{-90}{$\twoheadrightarrow$}}}

\newcommand{\da}{\downarrow}

\newcommand{\ra}{\rightarrow}

\newcommand{\bv}{\bigvee}
\newcommand{\bw}{\bigwedge}

\newcommand{\vd}{\vdash}
\newcommand{\dv}{\dashv}

\newcommand{\Ga}{\Gamma}

\newcommand{\si}{\sigma}

\newcommand{\vep}{\varepsilon}

\newcommand{\CC}{\mathcal{C}}

\newcommand{\CE}{\mathcal{E}}

\newcommand{\CP}{\mathcal{P}}

\newcommand{\CT}{\mathcal{T}}

\newcommand{\sQ}{\mathsf{Q}}

\newcommand{\sX}{{\sf X}}

\newcommand{\se}{{\sf e}}
\newcommand{\sm}{{\sf m}}

\newcommand{\sy}{{\sf y}}
\newcommand{\sfs}{\mathsf{s}}
\newcommand{\ssd}{{\sfs^\dag}}

\newcommand{\CPd}{{\CP^{\dag}}}

\newcommand{\QCat}{\sQ\text{-}{\sf Cat}}
\newcommand{\QDist}{\sQ\text{-}\mathsf{Dist}}

\newcommand{\PdA}{\CPd A}

\newcommand{\syd}{\sy^{\dag}}

\newcommand{\pleq}{\sqsubseteq}
\newcommand{\wti}{\widetilde}
\newcommand{\QInf}{\sQ\mbox{-}{\sf Inf}}
\newcommand{\QSup}{\sQ\mbox{-}{\sf Sup}}

\begin{document}

\maketitle
\begin{abstract}
 In contrast to the   fact that every completely distributive lattice is necessarily continuous in the sense of  Scott, it is shown that complete distributivity of a category enriched over the closed category obtained by endowing the  unit interval with a continuous t-norm does not imply its continuity in general. Necessary and sufficient conditions for the implication   are presented.
\end{abstract}


\section{Introduction}

Preordered sets are  often viewed as thin categories, and the other way around, categories have also been studied as ``generalized ordered structures''. Illuminating examples include the study of continuous categories in \cite{Adamek2003,Johnstone-Joyal} and  that of completely (totally) distributive  categories   in   \cite{Lucyshyn-Wright2012,Marmolejo2012}.  A bit more generally, categories enriched over a monoidal closed category can be viewed as ``ordered sets'' with \emph{truth-values}  taken in that closed category  \cite{Lawvere1973}. This point of view  has led to a theory of \emph{quantitative domains},  of which the core objects   are categories enriched in a quantale, see e.g.  \cite{Bonsangue1998,Flagg1996,Hofmann2011a,Hofmann2012,
KW2011,Wagner1994}.

Continuous dcpos  \cite{Gierz2003}  are characterized by the relation between  a poset $P$ and the poset  $\mathsf{Idl}(P)$ of ideals of $P$. For all $p\in P$, $\da\!p=\{x\in P\mid x\leq p\}$ defines an embedding
$\da\thinspace:P\to\mathsf{Idl}(P).$   A poset $P$ is   directed complete if $\da$ has a left adjoint $$\sup:\mathsf{Idl}(P)\to P $$ and   is continuous if there is a string of adjunctions
$$\thda\dv\sup\dv\thinspace\da\thinspace:P\to\mathsf{Idl}(P).$$

In a locally small category $\CE$, ind-objects (or equivalently, filtered colimits of representable  presheaves) play the role of ideals in posets. Let $\mathsf{Ind}\mbox{-}\CE$ be the category of all  filtered colimits of representable presheaves on $\CE$. Then, $\CE$ has   filtered colimits if the Yoneda embedding $\sy:\CE\to\mathsf{Ind}\mbox{-}\CE$ has a left adjoint
$$\colim:\mathsf{Ind}\mbox{-}\CE\to\CE $$ and it is further continuous if  there is a string of adjunctions $${\sf w}\dv\colim\dv\sy:\CE\to\mathsf{Ind}\mbox{-}\CE.$$

For   categories enriched in a commutative and unital quantale $\sQ$, forward Cauchy weights (i.e., presheaves generated by forward Cauchy nets, see Definition \ref{forward cauchy})  play the role of ideals. For each $\sQ$-category $A$, let $\CC A$ be the $\sQ$-category of all forward Cauchy weights of $A$. Then, $A$ is   forward Cauchy cocomplete if the  $\sQ$-functor $\se_A:A\to\CC A$, which is obtained by restricting the codomain of the enriched Yoneda embedding,   has a left adjoint
$${\sup}_A: \CC A\to A; $$
and $A$  is   continuous  if it is \emph{separated} (for definition see below) and there is a string of adjoint $\sQ$-functors $${\sf t}_A\dv{\sup}_A\dv\se_A:A\to\CC A.$$

If we replace,  in the definition of continuous lattice,   $\mathsf{Idl}(P)$ by the poset of all lower  sets,  then we obtain  the concept of (constructively) completely distributive lattices \cite{FW1990,Wood}. Similarly, if we replace, respectively, the category  of ind-objects and the $\sQ$-category  of forward Cauchy weights by the category of all small presheaves and the $\sQ$-category of all weights, then we  obtain the concepts of completely distributive categories \cite[Remark 4.7]{Marmolejo2012} and completely distributive $\sQ$-categories \cite{Stubbe2007,PZ15}.

\begin{center}
\begin{tabular}{l l l }
\hline
\{0,1\}  & $\sQ$  & {\sf Set} \\
\hline
preordered sets   & $\sQ$-categories  & categories  \\
lower sets  & weights & small  presheaves\\
directed lower sets & forward Cauchy weights & ind-objects \\

continuous dcpos & continuous $\sQ$-categories & continuous categories \\
c.d. lattices    &   c.d. $\sQ$-categories  & c.d. categories  \\
\hline
\end{tabular}
\end{center}

It is well-known in order theory that a completely distributive lattice  is necessarily continuous, see e.g. \cite{Gierz2003}. This paper investigates whether there is an enriched version of this conclusion. It turns out that it depends on the structure of the \emph{truth-values}, i.e., the structure of the closed category.  The main result, Theorem \ref{main},  shows that if $\sQ$ is the interval $[0,1]$  equipped with a continuous t-norm $\&$, then all completely distributive $\sQ$-categories are continuous if and only if  the corresponding implication (see below) \(\ra:[0,1]^2\to[0,1]\) is continuous at every point off the diagonal $\{(x,x)\mid x\in[0,1]\}$. 

\section{Complete quantale-enriched categories}

This section recalls some basic notions and results about quantale-enriched categories (which are a special case of quantaloid-enriched categories \cite{Stubbe2005}) that will be needed.

A commutative and unital  quantale (a quantale for short) \cite{Rosenthal1990} is a small, complete and symmetric monoidal closed category. Explicitly, a quantale
\[\sQ=(\sQ,\with,k)\]
is a commutative monoid with $k$ being the unit, such that the underlying set $\sQ$ is a complete lattice  and the multiplication $\with$ distributes over arbitrary joins. The unit $k$  need not be the top element in $\sQ$. If it happens that $k$ is the top element, then $\sQ$ is said to be \emph{integral}. The  multiplication $\&$ determines a binary operator $\ra$, often called the implication  corresponding to $\&$, via the adjoint property:
\[p\with q\leq r\iff q\leq p\ra r.\]

Let \(\sQ=(\sQ,\with,k)\) be   quantale.  A \emph{$\sQ$-category} consists of a set $A$ and a map $a:A\times A\to \sQ$  such that
 \[k\leq a(x,x)\quad \mbox{and}\quad
a(y,z)\with  a(x,y)\leq a(x,z)\]
 for all $x,y,z\in A$. We often write $A$ for the pair $(A, a)$ and $A(x,y)$ for $a(x,y)$ if no confusion would arise.

For a $\sQ$-category $A$, the underlying preorder $\pleq$ on $A$ refers to the preorder given by
$$x\pleq y\iff k\leq A(x,y).$$ Two elements $x,y$ in a $\sQ$-category $A$ are said to be isomorphic if $A(x,y)\geq k$ and $A(y,x)\geq k$. A $\sQ$-category $A$ is \emph{separated} if isomorphic elements of $A$ are necessarily identical. It is clear that $A$ is separated if and only if $(A,\pleq)$ is a partially ordered set.

\begin{exmp}   \label{d_L}  $(\sQ,d_L)$ is a separated $\sQ$-category, where $d_L(x,y)=x\ra y$  for all $x,y\in\sQ$.  \end{exmp}

A \emph{$\sQ$-functor} $f: A\to B$ between $\sQ$-categories is a map $f:A\to B$ such that $$A(x,y)\leq B(f(x),f(y))$$
for all $x,y\in A$. With the pointwise order between $\sQ$-functors inherited from $B$, i.e.,
$$f\leq g:A\to B \iff\forall x\in A,\ f(x)\pleq g(x),$$
$\sQ$-categories and $\sQ$-functors constitute  a  locally ordered category \[\QCat.\]

A \emph{$\sQ$-distributor} $\phi:A\oto B$ between $\sQ$-categories is a map $\phi:A\times B\to \sQ$ such that
\[B(y,y')\with \phi(x,y)\with  A(x',x)\leq\phi(x',y')\] for all $x,x'\in A$ and $y,y'\in B$. The composition of $\sQ$-distributors $\phi:A\oto B$ and $\psi:B\oto C$ is given by
\[\psi\circ\phi:A\oto C, \quad\psi\circ\phi(x,z)=\bv_{y\in B}\psi(y,z)\with \phi(x,y).\]
$\sQ$-categories and $\sQ$-distributors constitute a locally ordered category \[\QDist\] with   local order inherited from $\sQ$.

Each $\sQ$-functor $f: A\to B$ induces distributors
\[f_{*}(x,y)=B(f(x),y): A\oto B\quad\text{and}\quad f^{*}(y,x)=B(y,f(x)): B\oto A,\]
called respectively the \emph{graph} and \emph{cograph} of $f$.

Let $f: A\to B$ and $g:B\to A$ be $\sQ$-functors. We say $f$ is left adjoint to $g$ (or, $g$ is   right adjoint to $f$), and write $f\dashv g$, if
\[1_A\pleq g\circ f\quad\mbox{and}\quad f\circ g\pleq 1_B.\] It is easily seen that $f$ is left adjoint to $g$ if and only if \[f_*=g^*,\] i.e., $B(f(x),y)=A(x,g(y))$ for all $x\in A$ and  $y\in B$.

With $\star$ denoting the singleton $\sQ$-category with only one object and $a(\star,\star)=k$, $\sQ$-distributors of the form $\phi:A\oto\star$ are  called  \emph{weights}  (or, \emph{presheaves}) of $A$. The weights of $A$ constitute a $\sQ$-category $\CP A$ with
\[\CP A(\phi,\rho)=\bw_{x\in A}\phi(x)\ra\rho(x).\]
  Dually, $\sQ$-distributors  of the form $\psi:\star\oto A$ are called \emph{coweights} (or, \emph{copresheaves}) of $A$. The coweights of $A$ constitute a $\sQ$-category $\CPd A$ with
\[\CPd A(\psi,\si)=\bw_{x\in A}\si(x)\ra\psi(x).\]

For any $\sQ$-category $A$, the underlying order on $\CP A$ coincides with the local order in $\QDist$, while the underlying order on $\PdA$ is  \emph{opposite} to the local order in $\QDist$, i.e.,
$$\psi\pleq\si\ \text{in}\ \CPd A\iff\si\leq\psi\ \text{in}\ \QDist.$$

Each $\sQ$-functor $f:A\to B$ induces two $\sQ$-functors
\[\CP f:\CP A\to\CP B,\quad\CP f(\phi)=\phi\circ f^*\] and \[
\CPd f:\CPd A\to\CPd B,\quad\CPd f(\psi)=f_*\circ\psi.
\]
Both $\CP$ and $\CPd$ are endofunctors on $\QCat$. The functor $\CP$ can be made   into a monad $(\CP,\sfs,\sy)$, called  the \emph{presheaf monad}, with   unit   given by the Yoneda embedding
\[
\sy_A:A\to\CP A,\quad \sy_A(x)=A(-,x)
\]
and   multiplication   given by
\[ \label{s_def}
\sfs_A:\CP\CP A\to\CP A,\quad \sfs_A(\Phi)=\Phi\circ (\sy_A)_*:  A\oto \CP A\oto\star.
\]
Similarly, the functor $\CPd$ can also be made into a monad $(\CPd, \sfs^\dagger,\sy^\dagger)$, called  the \emph{copresheaf monad}, with unit $\sy^\dagger$   given by the co-Yoneda embedding
 \[
 \syd_A:A\to\CPd A,\quad \syd_A(x)=A(x,-)
 \]
 and   multiplication $\ssd$   given by
\[
\sfs_A^\dagger:\CPd\CPd A\to\CPd A,\quad \sfs_A^\dagger(\Psi)=(\syd_A)^*\circ\Psi:\star\oto\CPd A\oto A.
\]

The presheaf monad $(\CP,\sfs,\sy)$ is a KZ-doctrine
and the copresheaf monad $(\CPd,\ssd,\sy^\dagger)$ is a co-KZ-doctrine on the locally ordered category $\QCat$.\footnote{A  monad $(\CT,\sm,\se)$ on a locally ordered category $\sX$ is a KZ-doctrine (co-KZ-doctrine, resp.)  \cite{Monoidal top,Kock1995,Zo76} if $\CT$ is a 2-functor, and for each object  $A$ of $\sX$, there is a string of adjoint arrows \[\CT\se_A\dv\sm_A\dv\se_{\CT A}:\CT A\to\CT\CT A \quad ( \CT\se_A\vd\sm_A\vd\se_{\CT A}:\CT A\to\CT\CT A\text{, resp.}).\]   The latter condition is equivalent to \[\CT\se_A\leq\se_{\CT A} \quad (\CT\se_A\geq\se_{\CT A}\text{, resp.}) \] for each object  $A$ of $\sX$. }

Let $A$ be a $\sQ$-category and $\phi$ be a weight of $A$. An element $x$ of $A$ is called a \emph{supremum} of $\phi$ if for all $y\in A$,
\[A(x,y)=\CP A(\phi,\sy (y))=\bw_{z\in A}(\phi(z)\ra A(z,y)).\] In the terminology of category theory, the element $x$ is a \emph{colimit} of the identity functor $A\to A$ weighted by   $\phi$. However,  following the tradition of order theory, we call it a \emph{supremum of $\phi$} and denote it by \emph{$\sup_A\phi$}. Supremum of a weight $\phi$,   when exists, is unique   up to isomorphism.  We say that $A$ is \emph{cocomplete} \cite{Stubbe2005} if the Yoneda embedding   $\sy_A:A\to\CP A$ has a left adjoint, ${\sup}_A:\CP A\to A$, which sends each weight $\phi$ to its supremum.  Dually, we say that a $\sQ$-category $A$  is \emph{complete} if the co-Yoneda embedding $\syd_A:A\to\CPd A$ has a right adjoint,  $\inf_A:\CPd A\to A$, which sends   each $\psi\in\CPd A$ to its \emph{infimum}.

\begin{prop}{\rm(\cite[Proposition 5.10]{Stubbe2005})}
 A $\sQ$-category $A$ is complete if and only if it is cocomplete.
\end{prop}

Since $(\CP,\sfs,\sy)$ is a KZ-doctrine, a $\sQ$-category $A$ is a $\CP$-algebra if and only if $\sy_A:A\to\CP A$ has a left inverse (hence $A$ is separated), and in this case the left inverse  is necessarily a left adjoint of   $\sy_A$, see e.g. \cite[Theorem 2.4]{Hof2013}.   A \emph{$\CP$-homomorphism} $f:A\to B$ between $\CP$-algebras $A$ and $B$ is a $\sQ$-functor $f:A\to B$ such that \[{\sup}_B\circ \CP f=f\circ{\sup}_A,\] which is equivalent to   $f$ being a left adjoint. Therefore, the category of $\CP$-algebras and $\CP$-homomorphisms  is just the category \[\sQ\text{-}{\sf Sup}\] of separated  cocomplete $\sQ$-categories and left adjoint $\sQ$-functors. Dually, since $(\CPd,\ssd,\sy^\dagger)$ is a co-KZ-doctrine, a $\CPd$-algebra is exactly a separated complete $\sQ$-category;
a $\CPd$-homomorphism $f:A\to B$ between $\CPd$-algebras   is a right adjoint $\sQ$-functor. Thus, the category of $\CPd$-algebras and $\CPd$-homomorphisms is just the category \[\QInf\] of  separated complete $\sQ$-categories and right adjoint $\sQ$-functors.

For each subset $C$ of a $\sQ$-category $A$, we  write $\bv C$ for a join of $C$ (which is unique up to isomorphism) in $(A,\pleq)$;    likewise  we  write $\bw C$ for a meet of $C$ in $(A,\pleq)$.

\begin{prop} {\rm(\cite{Borceux1994a,Stubbe2006})} \label{coco by tens}
 A $\sQ$-category $A$ is cocomplete if and only if it satisfies the following conditions:
\begin{enumerate}[label=\rm(\arabic*)]\setlength{\itemsep}{0pt}
\item $A$ is tensored in the sense that for all $p\in \sQ$, $x\in  A$, there is an element $p\otimes x\in A$, called the tensor of $p$ with $x$, such that for any $y\in A$,
\[ A(p\otimes x,y)=p\ra A(x,y);\]
\item Every subset $C$ of $A$ has a join in $(A,\pleq)$  and  for all $x\in A$, \[A(\bv C, x)=\bw_{y\in C}A(y,x).\]
\end{enumerate}
\end{prop}

It is not hard to check that   for each weight $\phi$ of a cocomplete $\sQ$-category $A$, \[{\sup}_A\phi=\bv_{x\in X}\phi(x)\otimes x.\]
\begin{prop}{\rm(\cite{Borceux1994a,Stubbe2006})}\label{comp by cote}
 A $\sQ$-category $A$ is complete if  and only if it satisfies the following conditions:
\begin{enumerate}[label=\rm(\arabic*)]\setlength{\itemsep}{0pt}
\item  $A$ is cotensored in the sense that for all $p\in \sQ$, $x\in  A$, there is an element $p\multimap x\in  A$, called the cotensor of $p$ with $x$, such that for any $y\in A$,
\[A(y,p\multimap x) = p\ra A(y,x);\]
\item Every subset $C$ of $A$ has a meet in $(A,\pleq)$   and  for all $x\in A$, \[A(x,\bw C)=\bw_{y\in C}A(x,y).\]
\end{enumerate}
 \end{prop}

\begin{prop}{\rm(\cite{Stubbe2006})} \label{ad func} Let $f:A\to B$ be $\sQ$-functor between complete  $\sQ$-categories.
\begin{enumerate}[label=\rm(\arabic*)]\setlength{\itemsep}{0pt}
\item $f$ is a left adjoint if and only if $f(p\otimes x)=p\otimes f(x)$ for all $p\in \sQ$, $x\in A$ and $f(\bv C)=\bv f(C)$ for  all   $C\subseteq A$.
\item $f$ is a right adjoint if and only if $f(p\multimap x)=p\multimap f(x)$ for all $p\in \sQ$, $x\in A$ and $f(\bw C)=\bw f(C)$ for  all   $C\subseteq A$.
\end{enumerate}
\end{prop}

\begin{exmp}\label{exmp of comp}
(1) For each $\sQ$-category $A$, $\CP A$ is complete, in which \[(p\otimes\phi)(x)=p\with\phi(x)\quad \text{and}\quad (p\multimap\phi)(x)=p\ra\phi(x)\] for all $p\in\sQ$ and $\phi\in\CP A$.

(2) For each $\sQ$-category $A$, $\CPd A$ is complete, in which \[(p\multimap\psi)(x)=p\with\psi(x)\quad \text{and}\quad (p\otimes\psi)(x)=p\ra\psi(x)\] for all $p\in\sQ$ and $\psi\in\CPd A$.
\end{exmp}

\section{Completely distributive quantale-enriched categories}

A \emph{saturated class of weights} \cite{AK,KS2005,Stubbe2010} is a full submonad  $(\CT,\sm,\se)$ of the presheaf monad $(\CP,\sfs,\sy)$ on $\QCat$. Explicitly, a saturated class of weights is a triple $(\CT,\sm,\se)$  subject to:
\begin{itemize}\setlength{\itemsep}{0pt}
\item $\CT $ is a subfunctor of $\CP:\QCat\to\QCat $;
\item all inclusions $\vep_A:\CT A\to\CP A$ are fully faithful;
\item all $\vep_A$ form a natural transformation such that
\[\sfs\circ(\vep*\vep)=\vep\circ\sm\quad\mbox
{and}\quad\vep\circ\se=\sy.\]
\end{itemize}
Said differently, a saturated class of weights is a functor $\CT:\QCat\to\QCat$ such that   $\CT A$ is a full sub-$\sQ$-category of $\CP A$ through which the Yoneda embedding $\sy_A:A\to\CP A$  factors, and that for each $\Phi\in\CT\CT A$, the supremum of  \[\Phi\circ\varepsilon_A^*:\CP A\oto\CT A\oto\star\] in $\CP A$ belongs to $\CT A$.

Since  $(\CP,\sfs,\sy)$ is a KZ-doctrine, then so is every saturated class of weights $(\CT,\sm,\se)$ on $\QCat$. Thus, for each saturated class of weights $(\CT,\sm,\se)$  on $\QCat$, a \emph{$\CT$-algebra} $A$ is   a $\sQ$-category $A$ such that \[\se_A:A\to\CT A\] has a left inverse (which is necessarily a left adjoint of $\se_A$)
\[{\sup}_A :\CT A\to  A. \] Said differently, a $\CT$-algebra is a separated $\sQ$-category $A$ such that every $\phi\in\CT A$ has a supremum.
 A \emph{$\CT$-homomorphism} $f:A\to B$ between $\CT$-algebras   is a $\sQ$-functor such that
 \[f\circ{\sup}_A={\sup}_B\circ\CT f.\] The category of $\CT$-algebras and $\CT$-homomorphisms is denoted by \[\CT\text{-}{\sf Alg}.\]
For the largest saturated class of weights $\CP$, the category $\CP$-{\sf Alg} is just the category $\QSup$ of separated cocomplete $\sQ$-categories and left adjoint $\sQ$-functors.

It is clear that every $\CP$-algebra is  a $\CT$-algebra and every $\CP$-homomorphism is a $\CT$-homomorphism, so  the category $\CP$-{\sf Alg} is a  subcategory of $\CT\text{-}{\sf Alg}$.

\begin{prop} \label{retract of T-alg} Let $\CT$ be a saturated class of weights on $\QCat$. Then, every retract of a $\CT$-algebra in  $\QCat$ is a $\CT$-algebra. \end{prop}

\begin{proof}Suppose that $B$ is a $\CT$-algebra;  $s:A\to B$ and $r:B\to A$ are $\sQ$-functors such that $r\circ s=1_A$. Let $\sup_A$ be the composite \[\CT A\to^{\CT s}\CT B\to^{{\sup}_B}B\to^r A.\] Then   \[{\sup}_A\circ\se_A= r\circ{\sup}_B\circ\CT s\circ \se_A  = r\circ{\sup}_B\circ\se_B\circ s =r\circ s =1_A,\]   so, $\sup_A$ is a left inverse of $\se_A$ and consequently, $A$ is a $\CT$-algebra. \end{proof}

\begin{defn}\label{T-continuous} Let $(\CT,\sm,\se)$ be a saturated class of weights on $\QCat$. A $\sQ$-category is said to be $\CT$-continuous if it is a $\CT$-continuous $\CT$-algebra; that is,  if $A$ is separated  and there is a string of adjoint $\sQ$-functors
\[{\sf t}_A\dv{\sup}_A\dv\se_A:A\to\CT A.\]
\end{defn}

\begin{prop}\label{TA is T-continuous} Let $(\CT,\sm,\se)$ be a saturated class of weights on $\QCat$. Then,  for every $\sQ$-category $A$, the $\sQ$-category $\CT A$ is $\CT$-continuous.
\end{prop}

\begin{proof} Since $(\CT,\sm,\se)$ is saturated, it follows that for every $\sQ$-category $A$, there is a string of adjoint $\sQ$-functors \[\CT\se_A\dv\sm_A\dv\se_{\CT A}:\CT A\to\CT\CT A,\] which entails that $\CT A$ is $\CT$-continuous. \end{proof}

\begin{prop}\label{retract of continuous} Let $\CT$ be a saturated class of weights on $\QCat$. Then,  in the category  $\CT$-{\sf Alg}, every retract  of a $\CT$-continuous $\CT$-algebra is $\CT$-continuous.
\end{prop}
\begin{proof} Suppose that $B$ is a $\CT$-continuous $\CT$-algebra;  $s:A\to B$ and $r:B\to A$ are $\CT$-homomorphisms such that $r\circ s=1_A$. We claim that ${\sf t}_A:=\CT r\circ {\sf t}_B\circ s$ is left adjoint to $\sup_A$, hence $A$ is $\CT$-continuous. On one hand,
\begin{align*}{\sup}_A\circ {\sf t}_A&={\sup}_A\circ\CT r\circ {\sf t}_B\circ s\\
&=r\circ{\sup}_B\circ {\sf t}_B\circ s &(\mbox{$r$ is a $\CT$-homomorphism})\\
&=r\circ s\\
&=1_A.
\end{align*}
On the other hand,
\begin{align*}{\sf t}_A\circ {\sup}_A &=\CT r\circ {\sf t}_B\circ s\circ{\sup}_A\\
&=\CT r\circ {\sf t}_B\circ{\sup}_B\circ\CT s&(\mbox{$s$ is a $\CT$-homomorphism})\\
&\pleq\CT r\circ\CT s &({\sf t}_B\dashv{\sup}_B)\\
&=1_{\CT A}.
\end{align*}
Thus, ${\sf t}_A$ is left adjoint to $\sup_A$, as desired.
\end{proof}

\begin{cor}\label{continuity=retract} Let $\CT$ be a saturated class of weights. Then, a $\CT$-algebra $A$ is $\CT$-continuous if and only if it is a retract of   $\CT A$ in $\CT$-{\sf Alg}. \end{cor}

Letting $\CT=\CP$ in Definition \ref{T-continuous} we obtain the notion of completely distributive $\sQ$-categories. Explicitly,

\begin{defn} \cite{Stubbe2007} A  $\sQ$-category $A$ is   said to be completely distributive (or, totally continuous)
if it is a $\CP$-continuous $\CP$-algebra; that is, if $A$ is separated and there exists a string of adjoint $\sQ$-functors  \[{\sf t}_A\dv{\sup}_A\dv\sy_A: A\to\CP A.\]
  \end{defn}

\begin{prop}
A complete $\sQ$-category $A$ is completely distributive if and only if it is a retract of some power of $(\sQ,d_L)$ in   $\QSup$.
\end{prop}

\begin{proof} For each set $X$, the power  $(\sQ,d_L)^X$ (see Example \ref{d_L}) in $\QSup$ is clearly the $\sQ$-category $\CP X$ when $X$ is viewed as a discrete $\sQ$-category (defined in an evident way). So,  sufficiency follows from propositions   \ref{TA is T-continuous} and \ref{retract of continuous}. Necessity follows from the observation that a  completely distributive $\sQ$-category $A$ is a retract of $\CP A$  which is a retract of $\CP|A|$, where $|A|$ is the discrete $\sQ$-category with the same objects as those of $A$. \end{proof}

\begin{defn}\cite{Lai2018} A separated complete $\sQ$-category $A$ is   completely co-distributive if there exists a string of adjoint $\sQ$-functors:\[\syd_A\dashv{\inf}_A\dashv {\sf t}^\dagger_A:A\to\CPd A.\] \end{defn}

It is not hard to see that a $\sQ$-category $A$ is completely co-distributive if and only if $A^{\rm op}$, the opposite of $A$ given by $A^{\rm op}(x,y)=A(y,x)$, is   completely distributive.
Since $(\CPd,\ssd,\sy^\dagger)$ is a co-KZ-doctrine on   $\QCat$,   for each $\sQ$-category $A$, the $\sQ$-category  $\CPd A$ is  easily verified to be completely co-distributive.
It is known in lattice theory that the notion of complete distributivity is self dual, i.e., a complete lattice is completely distributive if and only if so is its opposite, see e.g. \cite[VII.1.10]{Johnstone1982}. But, this is not always true for $\sQ$-categories. In fact, it is proved in \cite[Theorem 8.2]{Lai2018} that for an integral quantale $\sQ$,  every completely distributive  $\sQ$-category is completely co-distributive if and only if   $\sQ$ satisfies the law of double negation. So, complete distributivity and complete co-distributivity are no longer equivalent concepts for quantale-enriched categories.

\section{Continuous quantale-enriched categories}
In order to define continuous $\sQ$-categories, the first step is to find  for $\sQ$-categories an analogue of ideals (= directed lower sets) in a partially ordered set  and/or ind-objects in a locally small category.
Forward Cauchy weights will play the role.

Let $A$ be a $\sQ$-category. A net in $A$ is a map $x$ from a directed set $(E,\leq)$ to $A$. It is customary to write $x_\lambda$ for $x(\lambda)$ and to write $\{x_\lambda\}_\lambda$ for the net.

\begin{defn}\cite{Bonsangue1998,Flagg1996,Wagner1997} \label{forward cauchy} Let $A$ be a $\sQ$-category. A net $\{x_\lambda\}_\lambda$ in  $A$ is called forward Cauchy if \[\bv_\lambda\bw_{\gamma\geq \mu\geq \lambda}A(x_\mu,x_\gamma)\geq k.\]
A weight $\phi:A\oto \star$ is called forward Cauchy   if
\[\phi =\bv_\lambda\bw_{\lambda\leq \mu}A(-,x_\mu) \] for some forward Cauchy net $\{x_\lambda\}_\lambda$ in $A$. \end{defn}

Let $\{x_\lambda\}_\lambda$ be a forward Cauchy net in a $\sQ$-category $A$. An element $x\in A$ is called a \emph{liminf} (a.k.a. Yoneda limit) \cite{Bonsangue1998,Wagner1997} of $\{x_\lambda\}_\lambda$,   if for all $y\in A$, \[ A(x,y)=\bv_\lambda\bw_{\mu\geq \lambda} A(x_\mu,y).\]

We say that a $\sQ$-category $A$ is \emph{forward Cauchy cocomplete} (a.k.a. Yoneda complete) if every forward Cauchy net has a liminf.
The following conclusion, which is proved in \cite[Lemma 46]{Flagg1996} when $\sQ$ is a \emph{value quantale} (see \cite[Definition 6]{Flagg1996}) and in \cite[Theorem 5.13]{Lai2007} for the general case, implies that a $\sQ$-category $A$ is forward Cauchy cocomplete if and only if every forward Cauchy weight of $A$ has a supremum.

\begin{prop} \label{yoneda limit = sup}  Let $\{x_\lambda\}_\lambda$ be a forward Cauchy net in a $\sQ$-category $A$. An element $x$ of $A$ is a liminf of  $\{x_\lambda\}_\lambda$ if and only if $x$ is a supremum of the weight \[\phi=\bv_\lambda\bw_{\mu\geq\lambda}A(-,x_\mu).\] \end{prop}

We do not know whether assigning to each $\sQ$-category $A$ the $\sQ$-category of forward Cauchy weights of $A$ gives a saturated class of weights, however, there is an easy-to-check sufficient condition which is presented in \cite{Flagg1996,Lai2007}.

A quantale is said to be \emph{continuous} if its underlying complete lattice is continuous. The following conclusion  is proved in \cite[Proposition 13]{Flagg1996} when $\sQ$ is a  value quantale  (which is necessarily integral and continuous)  and  in \cite[Theorem 6.5]{Lai2007} for the version stated below.

\begin{prop} \label{C is saturated} Let $\sQ$ be an integral and continuous quantale. Then,  assigning to each $\sQ$-category $A$   the $\sQ$-category \[\CC A:=\{\phi\in\CP A\mid \phi~\text{is forward Cauchy}\}\]
defines a saturated class of weights on $\QCat$, which is denoted by $\CC$.   \end{prop}

\begin{con} \label{convention} When talking about forward Cauchy weights, if not otherwise specified, we always assume that $(\sQ,\with ,k)$ is   continuous, commutative and integral. For such a quantale, the class $\CC$ of forward Cauchy weights is saturated and the category of $\CC$-algebras and $\CC$-homomorphisms is exactly the category of separated and forward Cauchy cocomplete $\sQ$-categories and  $\sQ$-functors that preserve liminf of forward Cauchy nets. \end{con}

Letting $\CT=\CC$ in Definition \ref{T-continuous} we obtain the notion of continuous $\sQ$-categories. Explicitly,

\begin{defn}\label{continuous Qcat} \cite{KW2011} A $\sQ$-category $A$  is  said to be  continuous if   it is a $\CC$-continuous $\CC$-algebra; that is, if $A$ is separated and there is a string of adjoint $\sQ$-functors: \[{\sf t}_A\dashv{\sup}_A\dashv \se_A:A\to\CC A.\] 
 \end{defn}

When $\sQ$ is the two-element  Boolean algebra $\{0,1\}$, a continuous $\sQ$-category is exactly a continuous dcpo.
It is well-known that a completely distributive lattice is necessarily a continuous lattice, so it is natural to ask:

\begin{ques} \label{ques} Is every completely distributive $\sQ$-category  continuous   in the sense of Definition \ref{continuous Qcat}? \end{ques}

As we shall  see  in Section \ref{last section}, the answer depends on the structure of the truth-values, i.e., the structure of the quantale $\sQ$. A sufficient and necessary condition will be given when $\sQ$ is  the   interval $[0,1]$ equipped with a continuous t-norm.

\begin{prop}The following statements are equivalent: \begin{enumerate}[label=\rm(\arabic*)]\setlength{\itemsep}{0pt}
\item Every completely distributive $\sQ$-category is continuous.
\item $\CP A$ is   continuous   for every $\sQ$-category $A$.
\end{enumerate}\end{prop}

\begin{proof} That $(1)$ implies $(2)$ is trivial. Conversely, let $A$ be a completely distributive $\sQ$-category. From Corollary \ref{continuity=retract} it follows that $A$ is a retract of $\CP A$ in $\QSup$, hence a retract of $\CP A$ in the category of $\CC$-algebras. Since $\CP A$ is continuous by assumption, then so is $A$ by Proposition \ref{retract of continuous}. \end{proof}


Given a cocomplete $\sQ$-category $A$, denote the set of all ideals of the complete lattice $(A,\pleq)$ by ${\sf Idl}(A)$. Since each ideal $D$ of $(A,\pleq)$ can be seen as a forward Cauchy net of $A$,   \[\Lambda(D):=\bv_{d\in D}A(-,d)\] is then a forward Cauchy weight. Conversely, given a forward Cauchy weight $\phi$ of $A$, \[\Ga(\phi):=\{x\in  A\mid\phi(x)\geq k\}\] is an ideal of $(A,\pleq)$.

 \begin{prop}\label{ideals and directed sets}
Let $A$ be a separated  complete $\sQ$-category. Then   \[\Lambda:({\sf Idl}(A),\subseteq)\to(\CC A,\leq)\] is a  left adjoint and a left inverse of \[\Ga:(\CC A,\leq)\to({\sf Idl}(A),\subseteq).\] Moreover, $\sup_A\phi=\bv\Ga(\phi)$ for each $\phi\in\CC A$.
 \end{prop}

\begin{proof} Suppose $D$ is an ideal in $(A,\pleq)$ and
$\phi$ is a forward Cauchy weight of $A$. Then
\[D\subseteq\Ga(\phi)\iff\forall d\in D,\phi(d)\geq k\iff
\Lambda(D)\leq\phi,\]
which implies $\Lambda\dashv \Ga$.

Now we check that for each   forward Cauchy weight $\phi$ of $A$, $\Lambda\Ga(\phi)=\phi$. On one hand, since $\Lambda$ is left adjoint to $\Ga$, it follows that $\Lambda\Ga(\phi)\leq\phi$. On the other hand, by assumption there is a forward Cauchy net $\{x_\lambda\}_{\lambda\in E}$  in $A$ such that \[\phi(x)=\bv_{\lambda\in  E}\bw_{\lambda\leq \mu} A(x,x_\mu).\]  Let \[D_\phi:=\Big\{\bw_{\mu\geq \lambda}x_\mu\mid \lambda\in E\Big\},\] where $\bw_{\mu\geq \lambda}x_\mu$ denotes the meet of $\{x_\mu\mid \mu\geq \lambda\}$ in the complete lattice $(A,\pleq)$.
Then $D_\phi$ is a directed subset of $(A,\pleq)$ and \(\phi(x)=\bv_{d\in D_\phi} A(x,d),\) so $D_\phi\subseteq \Ga(\phi)$, and consequently,   $\phi\leq\Lambda\Ga(\phi)$.

Finally, we check that  $\sup_A\phi=\bv\Ga(\phi)$ for each $\phi\in\CC A$. Since $\sup_A\phi=\bv_{x\in A}\phi(x)\otimes x$, it follows that $\sup_A\phi\geq\bv\Ga(\phi)$. Conversely, since \[{\sup}_A\phi={\sup}_A\big(\bv_{d\in D_\phi} A(-,d)\big)=\bv D_\phi,\] then $\sup_A\phi\leq\bv\Ga(\phi)$.
\end{proof}

\begin{cor}\label{Scott continuous} Let $A,B$ be separated complete $\sQ$-categories and $f: A\to B$  be a $\sQ$-functor. Then, $f$ preserves liminf of forward Cauchy nets if and only if \[f: (A,\pleq)\to (B,\pleq)\] is Scott continuous (i.e., preserves  directed joins).\end{cor}

\begin{lem}\label{CA closed under meets}
For each complete $\sQ$-category $A$, the set   $\CC A$ of forward Cauchy weights of $A$  is closed in $\sQ^A$ (pointwise ordered)  under   meets and directed joins.  \end{lem}

\begin{proof} Let $\{\phi_i\}_{i\in I}$ be a subset of $\CC A$. Then, for all $x\in A$, by continuity of $\sQ$  we have
\[\bw_{i\in I}\phi_i(x)=\bw_{i\in I}\bv_{d\in\Ga(\phi_i)}A(x,d)
=\bv_{s\in\prod\limits_{i\in I}\Ga(\phi_i)}\bw_{i\in I}A(x,s(i))
=\bv_{s\in\prod\limits_{i\in I}\Ga(\phi_i)}A\big(x,\bw_{i\in I}s(i)\big).\]
Since \[\left\{\bw_{i\in I}s(i)\mid s\in\prod\limits_{i\in I}\Ga(\phi_i)\right\}\] is a directed set of $(A,\pleq)$, it follows that $\bw_{i\in I}\phi_i$ is a forward Cauchy weight of $A$, hence belongs to $\CC A$.

Let $\{\phi_i\}_i $ be a directed set of $(\CC A,\leq)$. Then  $\{\Ga(\phi_i)\}_i $ is a directed family in $({\sf Idl}(A),\subseteq)$ and $D=\bigcup_i\Ga(\phi_i)$ is an ideal in $(A,\pleq)$. Since
\[\bv_{d\in D} A(-,d)=\bv_i\phi_i,\] it follows that $\bv_i\phi_i\in\CC A$.

Therefore, $\CC A$ is closed in $\sQ^A$  under meets and directed joins. \end{proof}

\begin{prop}\label{t in Q}
A separated complete $\sQ$-category $A$ is continuous if and only if  \begin{enumerate}[label=\rm(\arabic*)]\setlength{\itemsep}{0pt} \item $(A,\pleq)$ is a continuous lattice; \item for each $x\in A$ and each forward Cauchy weight $\phi$ of $A$, \[A(x,{\sup}_A\phi)=\bw_{y\ll x}\phi(y),\] where $\ll$ denotes the way below relation (see \cite{Gierz2003} for definition) in $(A,\pleq)$. \end{enumerate}
\end{prop}
\begin{proof}
Sufficiency follows from  the fact that, under the assumption, the assignment $x\mapsto\bv_{y\ll x}\sy(y)$ defines a left adjoint of $\sup_A:\CC A\to A$.

Now we turn to the necessity. Since $\CC A$ is closed in $\sQ^A$  under meets and directed joins, it is a continuous lattice since so is $\sQ$ by our convention. Since $(A,\pleq)$ is a retract of $(\CC A,\pleq)$ in the category of dcpos by Corollary \ref{Scott continuous}, it follows that $(A,\pleq)$ is also a continuous lattice. This proves (1).
As for (2), it suffices to check that if $A$ is continuous, then the left adjoint ${\sf t}_A:A\to\CC A$ of $\sup_A:\CC A\to A$ is given by ${\sf t}_A(x)=\bv_{y\ll x}\sy(y)$. On one hand,
since $x={\sup}_A \bv_{y\ll x}\sy(y)$, it follows that ${\sf t}_A(x)\leq \bv_{y\ll x}\sy(y)$.  On the other hand, since $\bv\Gamma({\sf t}_A(x))=\sup_A {\sf t}_A(x)=x$ by Proposition \ref{ideals and directed sets}, it follows that every element   way below $x$ in $(A,\pleq)$ is contained in $\Ga ({\sf t}_A(x))$, so $\bv_{y\ll x}\sy(y)\leq {\sf t}_A(x)$.
\end{proof}

\begin{prop}\label{cond impl cont} Let  $A$ be a  complete $\sQ$-category. If  $(A,\pleq)$ is a continuous lattice  and  for all $p\in \sQ$, the cotensor $p\multimap -:(A,\pleq)\to(A,\pleq)$ is Scott continuous, then $A$ is continuous.
\end{prop}

 \begin{proof} Since  $(A,\pleq)$ is a continuous lattice,   $A$ is necessarily separated. So it remains to show that for each $x\in A$ and each forward Cauchy weight $\phi$ of $A$, \[A(x,{\sup}_A\phi)=\bw_{y\ll x}\phi(y),\] where $\ll$ denotes the way below relation in $(A,\pleq)$.

 On one hand, since $\{y\in A\mid y\ll x\}$ is a directed set with join $x$, it follows that $\bv_{y\ll x}\sy(y)$ is a forward Cauchy weight with supremum $x$. Thus, \[A(x,{\sup}_A\phi)\geq \CC A\Big(\bv_{y\ll x}\sy(y),\phi\Big)=\bw_{y\ll x}\phi(y).\]

 On the other hand,   for all $p\in\sQ$,
 \begin{align*}
 p\leq A(x,{\sup}_A\phi)&\Longrightarrow x\leq p\multimap{\sup}_A\phi\\
 &\Longrightarrow x\leq\bv_{d\in \Ga(\phi)}(p\multimap d)& (\mbox{$\sup_A\phi=\bv\Ga(\phi)$})\\
 &\Longrightarrow \forall y\ll x, \exists d\in \Ga(\phi), y\leq p\multimap d & (\mbox{$(A,\pleq)$ is continuous})\\
 &\Longrightarrow \forall y\ll x, \exists d\in \Ga(\phi), p\leq A(y,d)\\
 &\Longrightarrow  \forall y\ll x, p\leq\bv_{d\in \Ga(\phi)} A(y,d)\\
 &\Longrightarrow \forall {y\ll x}, p\leq \phi(y)\\
 &\Longrightarrow p\leq\bw_{y\ll x}\phi(y),
 \end{align*}
hence  $A(x,\sup_A\phi)\leq\bw_{y\ll x}\phi(y)$.
\end{proof}

\begin{cor}  Let $(\sQ,\with,k)$ be an integral quantale such that   the underlying lattice   is completely distributive and that for all $p\in \sQ$, $p\with-:\sQ\to \sQ$ preserves filtered meets. Then, for every $\sQ$-category $A$, the $\sQ$-category $\CPd A$  is continuous.
\end{cor}

\begin{proof}
This follows immediately from a combination of Example \ref{exmp of comp}\thinspace(2), Proposition \ref{cond impl cont}, the fact that the underlying order of $\CPd A$ is opposite to that inherited from  $\sQ^A$  (pointwise ordered)  and that $\CPd A$ is closed in $\sQ^A$ under   joins and meets.
\end{proof}

Not all completely distributive $\sQ$-categories are continuous even when the underlying lattice of $\sQ$ is the interval $[0,1]$, as we shall see in Section \ref{last section}, so the following conclusion is a bit unexpected.

\begin{cor} Let $(\sQ,\with,k)$ be an integral   quantale such that   the underlying lattice   is completely distributive and that for all $p\in\sQ$,  $p\with-:\sQ\to \sQ$ preserves filtered meets. Then, every completely co-distributive    $\sQ$-category is continuous.
\end{cor}

 \begin{proof} Since $\inf:\CPd A\to A$ has both a left adjoint and a right adjoint, $A$ is a retract of $\CPd A$, which implies that $A$ is continuous because so is $\CPd A$ by the above corollary.
 \end{proof}

\section{Relation to distributive law}
Related to Question \ref{ques} a general one is: \begin{ques}
Let $(\CT,\sm,\se)$ be a saturated class of weights  on $\sQ$-{\sf Cat}.  Is every completely distributive $\sQ$-category   $\CT$-continuous? \end{ques}

The answer depends on whether the copresheaf monad $(\CPd,\ssd,\syd)$ \emph{distributes over} the monad $(\CT,\sm,\se)$. We would like to note that  in this section the quantale $\sQ$ is not assumed to be a continuous one as we have agreed in \ref{convention} when talking about forward Cauchy weights; that is to say, the results   apply to any commutative quantale.

By a \emph{lifting} of $(\CT,\sm,\se)$ through the forgetful functor $U:\sQ$-${\sf Inf}\to\sQ$-{\sf Cat} we mean a monad $(\widetilde{\CT}, \widetilde{\sm},\widetilde{\se})$ on $\sQ$-{\sf Inf} such that \[U\circ \widetilde{\CT} = \CT\circ U,~~ U\circ\widetilde{\sm}=\sm\circ U,~~ U\circ \widetilde{\se}=\se\circ U.\] \[\bfig
\square<600,370>[\QInf`\QInf`\QCat`\QCat;
\wti{\CT}`U`U`\CT]
\efig\]
It is clear that such a lifting of $(\CT,\sm,\se)$ exists if and only if  for each separated and complete $\sQ$-category $A$, both $\se_A:A\to\CT A$ and $\sm_A:\CT\CT A\to\CT A$ are   $\sQ$-{\sf Inf} morphisms. Furthermore, such a lifting, when it exists, is necessarily unique  since the   functor $U$ is injective on objects.

A \emph{distributive law} of the monad $\CPd$ over $\CT$ is a natural transformation $\delta:\CPd\CT\to\CT\CPd$   satisfying certain conditions, see e.g. \cite[II.3.8]{Monoidal top}. Since $\QInf$ is the category of Eilenberg-Moore algebras of the copresheaf monad $\CPd$, it follows from \cite[II.3.8.2]{Monoidal top} that distributive laws of $\CPd$ over $\CT$ correspond bijectively to  liftings of $(\CT,\sm,\se)$ through the forgetful functor $U$. Therefore,  distributive laws of $\CPd$ over $\CT$, when they exist, are unique. So in this case we simply say that $\CPd$ \emph{distributes over} $\CT$. The main result in this section asserts that for a saturated class of weights   $\CT$  on $\QCat$, to require that every completely distributive $\sQ$-category is $\CT$-continuous is to require that $\CPd$ distributes over $\CT$.

\begin{thm}\label{main1}  For a saturated class of weights   $\CT$  on $\QCat$, the following statements are equivalent:
\begin{enumerate}[label=\rm(\arabic*)]\setlength{\itemsep}{0pt}
\item Every completely distributive $\sQ$-category is $\CT$-continuous. \item The copresheaf monad $\CPd$ distributes over $\CT$.\end{enumerate}
\end{thm}

A lemma first.

\begin{lem}\label{comp mond} Let $\CT$ be a saturated class of weights on $\QCat$. Then, the following statements are equivalent:
\begin{enumerate}[label=\rm(\arabic*)]\setlength{\itemsep}{0pt} \item The copresheaf monad $\CPd$ distributes over $\CT$.
\item The composite $\CT\CPd$ is a monad on $\QCat$. \item $\CT$ has a lifting through the forgetful functor $U:\sQ$-${\sf Inf}\to\sQ$-{\sf Cat}.

\item For every separated complete $\sQ$-category $A$, $\CT A$ is a complete $\sQ$-category.
\item  For every separated complete $\sQ$-category $A$,
    the inclusion $\CT A\to\CP A$ has a left adjoint.
\end{enumerate}
\end{lem}
\begin{proof} The equivalence $(1)\Leftrightarrow (2)\Leftrightarrow(3)$ follows immediately from \cite[II.3.8.2]{Monoidal top}  and the fact that $\QInf$ is the category of Eilenberg-Moore algebras of the monad $\CPd$.

$(3)\Rightarrow(4)$ Obvious.

$(4)\Rightarrow(5)$ By Proposition \ref{ad func}, it suffices to check that $\CT A$ is closed in $\CP A$ with respect to   cotensors and   meets.
For   $p\in\sQ$ and $\phi\in\CT A$, let $p\multimap\phi$ be the cotensor of $p$ and $\phi$ in $\CT A$. Then, for all $x\in A$,  \[(p\multimap\phi)(x)=\CT A(\se_A (x),p\multimap\phi)=p\ra\CT A(\se_A (x),\phi)=p\ra\phi(x)\] by Proposition \ref{comp by cote}\thinspace(1), hence $\CT A$ is closed in $\CP A$ with respect to  cotensors.
If $\phi$ is the meet of a family $\phi_i$ in $\CT A$, then for all $x\in A$, \[\phi(x)=\CT A(\se_A (x),\phi)=\bw_i\CT A(\se_A (x),\phi_i)=\bw_i\phi_i(x)\] by Proposition \ref{comp by cote}\thinspace(2), hence $\CT A$ is closed in $\CP A$ with respect to meets.

$(5)\Rightarrow(3)$ For each object $A$ in $\QInf$,   since $\CP A$ is   complete   and  $\CT A$ is a retract of $\CP A$ in $\QCat$, it follows from  Proposition \ref{retract of T-alg}  that $\CT A$ is   complete. For each morphism $f:A\to B$ in $\QInf$, let $g:B\to A$ be the   left adjoint of $f$. Then, $\CP f:\CP A\to\CP B$ is right adjoint to $\CP g:\CP B\to\CP A$, so $\CT f:\CT A\to\CT B$ is a right adjoint because $\CT$ is a subfunctor of $\CP$.  Therefore, the assignment \[\wti{\CT} A:=\CT A\] gives rise to an endofunctor on $\QInf$. To see that $\wti{\CT}$ is  a  lifting of $\CT$ through the forgetful functor $U$, it remains to check that  for each separated complete $\sQ$-category $A$, both $\se_A:A\to\CT A$ and $\sm_A:\CT\CT A\to A$ are right adjoints. First, since $A$ is  separated and  cocomplete,   it is a $\CT$-algebra, so $\se_A$ is right adjoint to $\sup_A:\CT A\to A$. Second, since $\CT$ is a KZ-doctrine, it follows that $\sm_A$ is both a left   and a right adjoint.
\end{proof}

\begin{proof}[of Theorem \ref{main1}.] $(1)\Rightarrow(2)$ By Lemma \ref{comp mond}, it suffices to show that for each complete $\sQ$-category $A$, the inclusion functor $\CT A\to\CP A$ has a left adjoint. Since $\CP A$ is completely distributive,   then, by assumption,  the left adjoint $\sup_{\CP A}:\CT\CP A\to \CP A$ of $\se_{\CP A}:\CP A\to\CT\CP A$  has   a left adjoint, say,  ${\sf t}_{\CP A}:\CP A\to \CT\CP A$. Since $A$ is cocomplete, the Yoneda embedding $\sy_A:A\to\CP A$ has a left adjoint $\sfs_A:\CP A\to A$. Since any 2-functor preserves adjunctions,   it follows that $\CT \sfs_A:\CT\CP A\to\CT A$ is left adjoint to $\CT \sy_A :\CT A\to\CT\CP A$. Thus, $\CT \sfs_A\circ {\sf t}_{\CP A}:\CP A\to  \CT A$ is left adjoint to $\sup_{\CP A}\circ \CT \sy_A:\CT A\to \CP A$. Since $\CT$ is a submonad of $\CP$, then  \[{\sup}_{\CP A}\circ \CT \sy_A (\phi) =\sfs_{\CP A}\circ \CP \sy_A (\phi) =\phi  \] for all $\phi\in\CT A$.  Therefore, the inclusion functor  $\CT A\to\CP A$, which coincides with  ${\sup}_{\CP A}\circ \CT \sy_A $, has a left adjoint, given by $\CT \sfs_A \circ {\sf t}_{\CP A}$.

$(2)\Rightarrow(1)$ Let $A$ be a completely distributive $\sQ$-category. Since $\CPd$ distributes over $\CT$,  it follows from  Lemma \ref{comp mond} that the inclusion $\CT A\to\CP A$ has a left adjoint. Then, the composite of the left adjoint of $\sfs_A:\CP A\to A$ with the left adjoint of the inclusion $\CT A\to\CP A$ is   a left adjoint of $\sup_A:\CT A\to A$, so $A$ is $\CT$-continuous.  \end{proof}

\begin{rem} Putting $\CT=\CP$ in Theorem \ref{main1} one obtains that $\CPd$ distributes over $\CP$, as has already been pointed out in  \cite{Lai2017,Stubbe2017}. \end{rem}

\begin{prop} If $\CT$ is a saturated class of weights over which $\CPd$ distributes, then for each separated $\sQ$-category $A$, the following statements are equivalent:
\begin{enumerate}[label=\rm(\arabic*)]\setlength{\itemsep}{0pt}
\item $A$ is a $\wti{\CT}$-algebra, where $\widetilde{\CT}$ is the lifting of $\CT$  through the forgetful functor $ \sQ$-${\sf Inf}\to\sQ$-{\sf Cat}.
\item $A$ is a $\CT\CPd$-algebra.
\item $A$ is a   complete and $\CT$-continuous  $\sQ$-category.
\end{enumerate}
 \end{prop}

\begin{proof} The equivalence $(1)\Leftrightarrow(2)$ is a special case of a general result in category theory, see e.g. \cite[II.3.8.4]{Monoidal top}.  It remains to check   $(1)\Leftrightarrow(3)$. If $A$ is a $\wti{\CT}$-algebra, then  $\se_A:A\to\CT A$ has a left adjoint $\sup_A:\CT A\to A$ which is also a morphism in $\QInf$. This means that $\sup_A$ has a left adjoint, so $A$ is $\CT$-continuous. Conversely, let $A$ be a complete and $\CT$-continuous   $\sQ$-category. Then,  $\CT A$ is complete by Lemma \ref{comp mond}, so the string of adjoint $\sQ$-functors
\[{\sf t}_A\dv{\sup}_A\dv\se_A:A\to\CT A,\]
ensures that $\sup_A$ is a morphism in $\QInf$ and  consequently, $A$ is   a $\wti{\CT}$-algebra.
\end{proof}

\begin{cor}For each saturated class of weights $\CT$ on $\QCat$, let \(\CT\text{-}{\sf Cont}\) denote the category that has as objects  complete and $\CT$-continuous $\sQ$-categories  and has as morphisms those $\sQ$-functor that are  right adjoints and $\CT$-homomorphisms. If $\CPd$ distributes over $\CT$, then \(\CT\text{-}{\sf Cont}\) is  monadic over $\QCat$. \end{cor}

\section{The main result} \label{last section}

A continuous t-norm  \cite{Klement2000} is a continuous map $\with:[0,1]^2\to[0,1]$  that makes $([0,1],\with ,1)$ into a commutative quantale. Given a continuous t-norm $\with $, the quantale $\sQ=([0,1],\&,1)$  is clearly integral and continuous. We record here a simple fact  for later use: for any $x$ and $y$ in $[0,1]$,  $x$ is way below $y$ (i.e., $x\ll y$) if either $x=0$ or $x<y$.

\begin{exmp}Some basic   continuous t-norms and their implications:
\begin{enumerate}[label={\rm(\arabic*)}] 
\item  The G\"{o}del t-norm: \[ x\with y= \min\{x,y\}, \quad x\ra y=\begin{cases}
		1,&x\leq y,\\
		y,&x>y.
		\end{cases} \] The implication   $\ra$ of the G\"{o}del t-norm is  continuous  except at   $(x,x)$, $x<1$.

\item  The product t-norm:  \[ x\with_P y=xy, \quad x\ra y=\begin{cases}
		1,&x\leq y,\\
		y/x,&x>y.
		\end{cases} \] The implication   $\ra$
of the product t-norm is  continuous  except at   $(0,0)$. The quantale $([0,1], \with_P, 1)$ is  isomorphic to Lawvere's quantale $([0,\infty]^{\rm op},+,0)$ \cite{Lawvere1973}.

\item  The {\L}ukasiewicz t-norm:\[ x\with_{\L} y=\max\{0,x+y-1\}, \quad x\ra y=\min\{1-x+y,1\}. \] The implication   $\ra$
    of the {\L}ukasiewicz t-norm
    is   continuous on $[0,1]^2$.
	\end{enumerate}
\end{exmp}

Let $\with $ be a continuous t-norm. An element $p\in [0,1]$ is  said to be idempotent  if $p\with p=p$.

\begin{prop}{\rm(\cite[Proposition 2.3]{Klement2000})} \label{idempotent}
Let $\&$ be a continuous t-norm on $[0,1]$ and $p$ be an idempotent element of $\&$. Then $x\with y= \min\{x, y\} $ whenever $x\leq p\leq y$.
 \end{prop} It follows immediately  that $y\ra x=x$ whenever  $x< p\leq y$ for some idempotent   $p$. Another consequence of  Proposition \ref{idempotent} is that for any idempotent elements $p, q$    with $p<q$,  the restriction of $\with $ to $[p,q]$, which is also denoted by $\with$,  makes $[p,q]$ into a commutative quantale with $q$ being the unit element.    The following  theorem   is of fundamental importance in the theory of continuous t-norms.

\begin{thm} {\rm(\cite{Klement2000,Mostert1957})}
\label{ordinal sum} Let $\with $ be a continuous t-norm. If $a\in [0,1]$ is non-idempotent, then there exist idempotent elements $a^{-}, a^{+}\in [0,1]$ such that $a^-<a<a^+$ and that the quantale  $([a^{-},a^{+}],\with ,a^{+})$ is  isomorphic to either $([0,1],\with_{\L},1)$ or $([0,1],\with_{P},1)$. Conversely, for each   set of disjoint open intervals $\{(a_n,b_n)\}_n$ of $[0,1]$, the binary operator \[x\with y:=\begin{cases}a_n+(b_n-a_n)T_n\Big(\displaystyle{\frac{x-a_n}{b_n-a_n}, \frac{y-a_n}{b_n-a_n}}\Big), & (x,y)\in [a_n,b_n]^2,\\
\min\{x,y\}, & {\rm otherwise} \end{cases}\] is a continuous t-norm, where each $T_n$ is a continuous t-norm on $[0,1]$. \end{thm}

Let $\with$ be  a continuous t-norm and $\sQ=([0,1],\&,1)$. Then from Proposition \ref{C is saturated} we obtain that the class  $\CC$ of  forward Cauchy weights is   saturated.
Now we   present the main result in this paper.

\begin{thm}\label{main} Let $\with$ be  a continuous t-norm and $\sQ=([0,1],\&,1)$. Then the following statements are equivalent:
\begin{enumerate}[label=\rm(\arabic*)]\setlength{\itemsep}{0pt}
\item Every completely distributive   $\sQ$-category is continuous.
\item The  $\sQ$-category $([0,1],d_L)$ is continuous.
\item For each non-idempotent element $a\in[0,1]$, the quantale $([a^-,a^+],\with ,a^+)$ is isomorphic to $([0,1],\with_P,1)$   whenever $a^->0$.
\item The implication   $\ra:[0,1]^2\to[0,1]$ is  continuous  at every point off the diagonal $\{(x,x)\mid x\in[0,1]\}$.
\item For each $p\in(0,1]$, the map $p\ra -:[0,1]\to[0,1]$ is Scott continuous  on $[0,p)$.
\item For every  complete $\sQ$-category $A$, the inclusion $\CC A\to\CP A$ has a left adjoint.
\item   $\CPd$ distributes over $\CC$.
\end{enumerate}
\end{thm}

\begin{proof}
$(1)\Rightarrow(2)$  Obvious.

$(2)\Rightarrow(3)$
By Proposition \ref{t in Q}, if the $\sQ$-category $([0,1],d_L)$ is continuous, then  for each $x\in[0,1]$ and each forward Cauchy weight $\phi$ of $([0,1],d_L)$, \[x\ra\sup\phi=\bw_{y\ll x}\phi(y).\] Now, suppose on the contrary that  there exist idempotent elements $p,q>0$  such that $([p,q],\with ,q)$ is isomorphic to $([0,1],\with_{\L},1)$. Let $\phi$ be the forward Cauchy weight $\bv_{r<p}\sy(r)$. Then for all $x\in(p,q)$, \begin{align*}\bw_{y\ll x}\phi(y)& =\bw_{y< x}\bv_{r<p}(y\ra r) &(x>0)\\ & =\bw_{p\leq y< x}\bv_{r<p}(y\ra r)&(\text{$y\ra r=1$ whenever $y\leq r$})\\ &=\bv_{r<p}r &(\text{$y\ra r=r$ since}~r<p\leq y) \\ &=p.\end{align*} But, since $([p,q],\with ,q)$ is isomorphic to $([0,1],\with_{\L},1)$, it follows that
\[x\ra\sup\phi=x\ra p>p \] whenever $p<x<q$, a contradiction.

$(3)\Rightarrow(4)$ Routine verification.

$(4)\Rightarrow(5)$ Trivial.

$(5)\Rightarrow(6)$ It suffices to show that for every  complete $\sQ$-category $A$,  $\CC A$ is closed in $\CP A$ under meets and cotensors. That $\CC A$ is closed in $\CP A$ under meets is ensured by Lemma \ref{CA closed under meets}. To see that $\CC A$ is closed in $\CP A$ under cotensors, for $p\in[0,1]$ and $\phi\in\CC A$, set \[D:=\{d\in A\mid p\leq\phi(d)\}.\] Then $D$ is    a directed set of  $(A,\pleq)$. We claim that  $\{p\multimap y\mid y\in \Ga(\phi)\}\subseteq D$. In fact, since $\phi=\bv_{z\in\Gamma(\phi)}A(-,z)$, it follows that for all $y\in \Ga(\phi)$,
\begin{align*}p\ra\phi(p\multimap y)&=p\ra\bv_{z\in \Ga(\phi)} A(p\multimap y,z)\\ &\geq p\ra A(p\multimap y,y)\\ &= A(p\multimap y,p\multimap y)\\ &=1,\end{align*}
hence $p\multimap y\in D$.

Let \[\rho:=\bv_{d\in D} A(-,d).\] Since $\rho$ is a forward Cauchy weight,   it suffices to show that $p\ra\phi=\rho$. That $\rho\leq p\ra\phi$ is clear. It remains to check that $p\ra\phi(x)\leq\rho(x)$ for all $x\in A$. If $p\leq\phi(x)$, then  $x\in D$ and \[  \rho(x)=\bv_{d\in D} A(x,d)\geq A(x,  x) =1.\]
If  $p>\phi(x)$, then \begin{align*}p\ra\phi(x)&=p\ra\bv_{y\in \Ga(\phi)} A(x,y)\\ &
 =\bv_{y\in \Ga(\phi)}(p\ra A(x,y)) &\big(\bv_{y\in \Ga(\phi)} A(x,y)=\phi(x)<p\big) \\ &
 =\bv_{y\in \Ga(\phi)} A(x,p\multimap y) \\ &
 \leq\bv_{d\in D} A(x,d)\\ &
 =\rho(x).
\end{align*}

$(6)\Rightarrow(7)$  Lemma \ref{comp mond}.

$(7)\Rightarrow(1)$ Theorem \ref{main1}.
\end{proof}

By the ordinal sum decomposition theorem, the map $\&:[0,1]^2\to[0,1]$, given by \[x\with y:=\begin{cases}1/2+ \max\{x+y-3/2,0\},  & (x,y)\in [1/2,1]^2,\\
\min\{x,y\}, & {\rm otherwise}, \end{cases}\] is a continuous t-norm. Since the restriction of $\&$ on $[1/2,1]$ is isomorphic to the \L ukasiewicz t-norm, it follows that for the quantale $\sQ=([0,1],\&,1)$,  the completely distributive $\sQ$-category $([0,1], d_L)$ is not continuous.

\begin{rem} A continuous t-norm $\with $ is said to be Archimedean if it has no idempotent elements other than $0$ and $1$ \cite{Klement2000}. It is known that an Archimedean continuous t-norm is isomorphic to either the product t-norm or the \L ukasiewicz t-norm. If $\sQ=([0,1],\with,1)$ with $\with $ being a continuous Archimedean t-norm, then the converse conclusion of Proposition \ref{cond impl cont} is also true. That means,   if a complete $\sQ$-category $A$ is continuous, then for each $p\in[0,1]$, the map \[p\multimap-:(A,\pleq)\to(A,\pleq)\] is Scott continuous. Given a directed set $D$ of  $(A,\pleq)$, let \[\phi:=\bv_{d\in D} A(-,d).\] Since $\with $ is Archimedean,  the map \[p\ra-:[0,1]\to[0,1]\] is continuous  for all $p\in[0,1]$. It follows from the argument of $(5)\Rightarrow(6)$ in Theorem \ref{main} that $p\ra\phi\in\CC A$ and it is the cotensor of $p$ with $\phi$ in $\CC A$. Therefore,    \begin{align*}  p\multimap\bv D &=p\multimap{\sup}_A\phi\\  &={\sup}_A (p\ra\phi)  & \mbox{(${\sup}_A$ is a right adjoint)}\\  &={\sup}_A\bv_{d\in D}(p\ra  A(-,d)) & \mbox{($p\ra-$ is continuous)}\\  &={\sup}_A\bv_{d\in D}  A(-,p\multimap d) &(\text{definition of cotensor})\\  &=\bv_{d\in D} p\multimap d. \end{align*} \end{rem}



\end{document}